\documentclass[12pt]{amsart}

\usepackage{amssymb}

\usepackage[centering,body={6in,9in}]{geometry}

\usepackage{enumerate}

\usepackage{color}

\newtheorem{thm}{Theorem}[section]
\newtheorem{conj}[thm]{Conjecture}
\newtheorem{lem}[thm]{Lemma}

\newtheorem{defn}[thm]{Definition}
\numberwithin{equation}{section}
\newtheorem{cor}[thm]{Corollary}

\DeclareMathOperator\diam{diam}
\DeclareMathOperator\dist{dist}
\def\ubdim{\overline{\dim}_B}
\def\lbdim{\underline{\dim}_B}
\def\hdim{\dim_H}

\newcommand\N{\Bbb N}
\newcommand\Z{\Bbb Z}
\newcommand\Q{\Bbb Q}
\newcommand\R{\Bbb R}
\newcommand\T{\Bbb T}

\def\sm{[\!]}

\begin{document}

\title{Affine embeddings of Cantor sets and  dimension of $\alpha\beta$-sets}

\author{De-Jun Feng}
\address[De-Jun Feng]{Department of Mathematics, The Chinese University of
    Hong Kong, Shatin, Hong Kong}
\email{djfeng@math.cuhk.edu.hk}

\author{Ying Xiong}
\address[Ying Xiong]{Department of Mathematics, South China University of
    Technology, Guangzhou, 510641, P.~R.~China}
\email{xiongyng@gmail.com}


\begin{abstract}
  Let $E, F\subset \R^d$ be two self-similar sets, and suppose that $F$ can be affinely embedded into $E$.  Under the assumption that $E$ is dust-like and has a small Hausdorff dimension, we prove the logarithmic commensurability between the contraction ratios of $E$ and $F$. This gives a partial affirmative answer to Conjecture 1.2 in \cite{FHR14}. The proof  is based on our study of the box-counting dimension of a class of multi-rotation invariant sets on the unit circle, including the $\alpha\beta$-sets initially studied by Engelking and  Katznelson.
\end{abstract}

\maketitle

\section{Introduction}\label{sec:intro}

For $A,B\subset \R^d$, we say that $A$ can be {\it affinely embedded} into
$B$ if $f(A)\subset B$ for some affine map $f\colon \R^d\to \R^d$ of the form
$f(x)=Mx+a$, where $M$ is an invertible $d\times d$ matrices and $a\in \R^d$.
In this paper, we investigate the necessary conditions under which one
self-similar set can be  affinely embedded into another self-similar  set.

Before formulating our result, we
first recall some terminologies about self-similar sets. Let
$\Phi=\{\phi_i\}_{i=1}^\ell$ be an {\it  iterated function system} (IFS) on
$\R^d$, that is,  a finite family of contractive mappings on $\R^d$. It is well known (cf. \cite{Hutch81}) that  there is a unique
non-empty compact set $K\subset \R^d$, called the {\it attractor} of $\Phi$,
such that 
\[K=\bigcup_{i=1}^\ell \phi_i(K).\] 
Correspondingly, $\Phi$ is called a {\it generating IFS} of $K$. We say that
$\Phi$ satisfies the {\it open set condition} (OSC)  if there exists a non-empty
bounded open set $V\subset \R^d$ such that $\phi_i(V)$, $1\leq i\leq \ell$,
are pairwise disjoint subsets of $V$. Similarly, we say that $\Phi$ satisfies
the {\it strong separation condition} (SSC) if $\phi_i(K)$ are pairwise disjoint
subsets of $K$. The strong separation condition  always implies the open set
condition (\cite{Hutch81}). When all maps in an IFS $\Phi$ are similitudes,
the attractor $K$ of $\Phi$ is called a {\it self-similar} set. By a similitude we mean  a map $\phi:\R^d\to \R^d$ of the form $\phi(x)=\rho P x+a$, with $\rho>0$, $a\in \R^d$ and $P$ an $d\times d$ orthogonal matrix. A self-similar set is called {\it nontrivial} if it is not a singleton.

The problem of determining whether one self-similar set can be affinely
embedded into another self-similar set was first studied in
\cite{FHR14}, revealing some interesting connections to smooth embeddings
and intersections of Cantors sets. It was shown \cite{FHR14} that, under
the open set condition, \footnote{Here we say that a self-similar set
satisfies the open set condition if it has a generating IFS which satisfies
this condition.} one nontrivial self-similar set $F$ can be embedded into another
self-similar set $E$ under a $C^1$-diffeomorphism if and only if it can be
affinely embedded into $E$; moreover, if $F$ can not be affinely embedded
into $E$, then there is a dimension drop in the intersection of $E$ and any
$C^1$-image of $F$ in the sense that
\[\hdim (E\cap f(F))<\min\{\hdim E,\;\hdim F\},\] 
where $f$ is any $C^1$-diffeomorphism on $\R^d$, and $\dim_H$ stands for Hausdorff dimension (cf. \cite{Fal-book, Mat-book}).

The above affine embedding problem is also closely related to  other investigations
on self-similar sets and measures, including classifications of self-similar subsets of Cantor sets \cite{FeRaW15},   structures of generating IFSs of Cantor sets~\cite{FeWa09, DeLa13, DeLa16},  Hausdorff dimension of
intersections of Cantor sets~\cite{ElKeM10,Furst70},  Lipschitz equivalence and Lipschitz
embedding of Cantor sets \cite{FalMa92, DWXX11},    geometric
rigidity of $\times m$-invariant measures~\cite{Hochm12}, and equidistribution from fractal measures \cite{HoSh15}.

It is natural to expect that, if one nontrivial  self-similar set can be affinely embedded into another  self-similar set which is totally disconnected, then the contraction ratios of these two sets should satisfy certain arithmetic relations. The following conjecture has been formulated from this view point.

\begin{conj}[\cite{FHR14}]
\label{conj-1}

  Suppose that $E,F$ are two totally disconnected nontrivial self-similar sets  in $\R^d$, generated by IFSs $\Phi=\{\phi_i\}_{i=1}^\ell$ and $\Psi=\{\psi_j\}_{j=1}^m$ respectively. Let $\rho_i,\gamma_j$ denote the contraction ratios of $\phi_i$ and $\psi_j$. Suppose that  $F$ can be affinely
  embedded into $E$. Then for each $1\le j\le m$, there exist non-negative
  rational numbers $t_{i,j}$ such that
  $\gamma_j=\prod_{i=1}^\ell\rho_i^{t_{i,j}}$. In particular, if $\rho_i=\rho$ for all $1\leq i\leq \ell$, then $\log \gamma_j/\log \rho\in \Q$ for $1\leq j\leq m$.
\end{conj}

We remark that the above arithmetic relations on $\rho_i, \gamma_j$ do fulfil when $E$ and $F$ are dust-like (i.e., $\Phi$ and $\Psi$ satisfy the SSC) and  Lipschitz equivalent \cite{FalMa92}. Nevertheless, no arithmetic conditions are needed for the Lipschitz embeddings. Indeed, it was shown in \cite{DWXX11} that if $E,F$ are dust-like with $\dim_HF<\dim_HE$, then $F$ can be Lipschitz embedded into $E$.

So far  Conjecture  \ref{conj-1} has been considered in \cite{FHR14, Alg16, Shm16, Wu16} in the special case that  $\Phi$ is homogeneous, that is, $\rho_i=\rho$ for all $i$.
  It was proved in \cite{FHR14} that the conjecture  is true under the additional assumptions that $\Phi$ is homogeneous satisfying  the SSC and $\dim_HE<1/2$.
Recently,  Algom \cite{Alg16} showed that in the case that $d=1$,  the conjecture holds  under the SSC and homogeneity  on $\Phi$,  the OSC on $\Psi$  and an additional assumption that $\dim_HE-\dim_HF<\delta$, where $\delta$ is a positive constant depending on $\dim_HF$.  Very recently, Shmerkin \cite{Shm16} and Wu \cite{Wu16} independently obtained much sharper result  in the  case that $d=1$.  Shmerkin \cite{Shm16} proved that   Conjecture \ref{conj-1} holds  under the assumptions that $d=1$,   $\Phi$ is homogeneous satisfying the OSC and $\dim_HE<1$. Wu \cite{Wu16} proved the conjecture under almost the same assumptions, except for putting  the SSC on $\Phi$ instead of the OSC.

In this paper we consider the  general case that $\Phi$ might not be homogeneous. Let $\Q$ denote the set of rational numbers. For $u_1,\ldots, u_k\in \R$, set
\begin{equation*}
\label{e-uu}
\mbox{span}_\Q(u_1,\ldots, u_k)=\left\{\sum_{i=1}^k t_iu_i:\; t_i\in \Q\right\}.
\end{equation*}
Then $\mbox{span}_\Q(u_1,\ldots, u_k)$ is  a linear space over the field $\Q$ with dimension $\leq k$.

 Our main result is the following.

\begin{thm}\label{thm-2}
 Under the assumptions of Conjecture \ref{conj-1}, suppose in addition that $\Phi$ satisfies the SSC and $\hdim E<c$, where
  \begin{equation}
  \label{e-1.1}
  c=\left\{
  \begin{array}{ll}
   1/4, & \mbox{ if }\ell=2,\\
  1/4,    & \mbox{ if } \ell\geq 3, \; \lambda= 1,\\
  1/(2\lambda+2), & \mbox{ if } \ell\geq 3, \; \lambda> 1,
  \end{array}
  \right.
   \end{equation}
with $\lambda=\dim \mbox{span}_\Q(\log \rho_1,\ldots, \log \rho_\ell)$. Then
the conclusion of Conjecture \ref{conj-1} holds.
\end{thm}

The proof of Theorem \ref{thm-2} is based on our study  of the box counting
dimension  of certain multi-rotation invariant sets on the unit circle.  To
be more precise, we first introduce some notation and definitions.  Let
$\T=\R/\Z$ denote the unit circle (which can be viewed as the unit interval
$[0,1]$ with the endpoints being identified). Let $\pi:\R\to \T$ be the
canonical mapping defined by $x\mapsto \{x\}$, where $\{x\}$ stands for the
fractional part of $x$.

\begin{defn}
\label{defn-1}
Let $\alpha_1,\dots,\alpha_\ell\in \R$ with $\ell\geq 2$. A non-empty closed set $K\subset \T$ is called an $(\alpha_1,\ldots, \alpha_\ell)$-set
if
\[
K\subset \bigcup_{i=1}^\ell (K-\pi(\alpha_i))
\]
equivalently if, whenever $x\in K$, then there exists $i\in \{1,\ldots,
\ell\}$ so that $x+\pi(\alpha_i)\in K$. Moreover, a sequence $(x_n)_{n=0}^\infty$ of points in $\T$ is called an $(\alpha_1,\ldots, \alpha_\ell)$-orbit if $$x_{n+1}-x_n\in \{\pi(\alpha_1),\ldots, \pi(\alpha_\ell)\}$$  for all $n\geq 0$.
\end{defn}

\begin{defn}
\label{defn-2}
Let $\alpha_1,\ldots, \alpha_\ell\in \R$ with $\ell\geq 1$. Say that
$\alpha_1,\ldots,\alpha_\ell$ are $\Q_+$-independent $(\mbox{mod } 1)$ if the
following equation
\[
 t_1\alpha_1+\ldots+t_\ell \alpha_\ell\equiv 0 \quad (\mbox{\rm mod }1)
\]
in the variables $t_1,\ldots, t_\ell$ has a unique solution $(0,\ldots, 0)$
in $\Q_+^\ell$, where $\Q_+$ stands for the set of non-negative rational
numbers.
\end{defn}

Similarly we can define $\Q$-independence $(\mbox{mod } 1)$ via replacing $\Q_+$ by $\Q$ in Definition~\ref{defn-2}. It is clear that the $\Q$-independence $(\mbox{mod } 1)$ implies
the $\Q_+$-independence $(\mbox{mod } 1)$.

The study of $(\alpha_1,\ldots,\alpha_\ell)$-sets has its origin in the early works of Engelking and Katznelson \cite{Eng61, Kat79}. In 1961,  Engelking \cite{Eng61} raised the question of existence of nowhere dense $(\alpha, \beta)$-sets (for short, $\alpha\beta$-sets), where $\alpha,\beta$ are $\Q$-independence $(\mbox{mod } 1)$.  Finally in 1979,  Katznelson \cite{Kat79} gave an affirmative answer to this question.  He showed that for any such pair $(\alpha,\beta)$, there always exist nowhere dense $\alpha\beta$-sets;  furthermore  for certain special  pairs $(\alpha, \beta)$, there  exist $\alpha\beta$-sets of  Hausdorff dimension $0$.

In contrast to Katznelson's result, we prove the following result  claiming that, any  $(\alpha_1,\ldots, \alpha_\ell)$-orbit passing through  infinitely many points has a large lower box-counting dimension (cf. \cite{Fal-book, Mat-book} for the definition).
\begin{thm}
\label{thm-5}
Let $\alpha_1,\ldots, \alpha_\ell\in \R$ with $\ell\geq 2$.  Suppose that  $(x_n)_{n=0}^\infty$ is an $(\alpha_1,\ldots, \alpha_\ell\}$-orbit passing through  infinitely many points. Let $K$ be the closure of the set $\{x_n:\; n\geq 0\}$.   Then the following statements hold.
\begin{itemize}
\item[(i)] If $\ell=2$, then either $K-K=\T$ or $K$ has a non-empty interior; in particular, $$\underline{\dim}_BK\geq 1/2,$$ where  $\underline{\dim}_B$ stands for lower box-counting dimension.
\item[(ii)] If $\ell\geq 2$, then $$\underline{\dim}_BK\geq
\left\{
\begin{array}{ll}
1,& \mbox{ if }r=1,\\
1/(r+1), & \mbox{ if }r>1,
\end{array}
\right.
$$
 where $r=\dim \mbox{span}_\Q(1,\alpha_1,\ldots,\alpha_\ell)-1$.
\end{itemize}
 \end{thm}

Notice that when $ \alpha_1,\ldots, \alpha_\ell$ are $\Q_+$-independent $(\mbox{mod } 1)$,   $x_n\neq x_m$ for different $n,m$ for any $(\alpha_1,\ldots, \alpha_\ell)$-orbit $(x_n)_{n=0}^\infty$.  Hence by  Theorem \ref{thm-5},  we have the following corollary,  saying that under the assumption of $\Q_+$-independence,  every $\alpha\beta$-set or more generally, every $(\alpha_1,\ldots, \alpha_\ell)$-set has a large lower box-counting dimension.

\begin{cor}
\label{cor-1}
Let $\alpha_1,\ldots, \alpha_\ell\in \R$ with $\ell\geq 2$. Assume that $\alpha_1,\ldots, \alpha_\ell$ are $\Q_+$-independent $(\mbox{mod } 1)$. Let $K\subset \T$ be an $(\alpha_1,\ldots, \alpha_\ell)$-set. Then the statements (i), (ii) listed in Theorem \ref{thm-5} hold for $K$.
 \end{cor}

To our best  knowledge, Theorem \ref{thm-5} seems to be new. It  not only plays a key role in our proof of Theorem \ref{thm-2}, but is also interesting in its own right.

This paper is organized as follows. In Section~\ref{sec:MIR} we prove Theorem \ref{thm-5}. In Section~\ref{sec:proof} we prove Theorem \ref{thm-2}. In Section~\ref{sec:4}, we pose several questions for further study.
\section{Box-counting dimension of multi-rotation invariant sets}
\label{sec:MIR}

In this section, we prove Theorem \ref{thm-5}.  Let $\ell\in\N$, $\ell\geq 2$ and $\alpha_1,\dots,\alpha_\ell\in \R$.  Suppose that $(x_n)_{n=0}^\infty$ is an $(\alpha_1,\ldots, \alpha_\ell)$-orbit that takes infinitely many values.   Without loss of generality, we assume that $x_0=0$. Then by Definition~\ref{defn-1}, there exists a sequence $(\omega_n)_{n=1}^\infty$ with $\omega_n\in \{1,\ldots, \ell\}$ such that
\begin{equation}
\label{e-xn'}
x_n\equiv \sum_{i=1}^n\alpha_{\omega_i}\pmod 1,\quad n=1,2,\ldots.
\end{equation}

 Set $X=\{x_n:\; n\geq 0\}$.  Then $K=\overline{X}$, where $\overline{X}$ stands for the closure of $X$.   Below we prove parts (i) and (ii) of Theorem~\ref{thm-5} separately.

\begin{proof}[Proof of Theorem~\ref{thm-5}(i)]
Assume that $\ell=2$.  It is enough to show that either $X-X$ is  dense in $\T$, or  $\overline{X}$ has a non-empty interior.  As a direct consequence,
  $$
 2\underline{\dim}_BK= 2\underline{\dim}_B X\ge \underline{\dim}_B(X-X)=1,
 $$
where the second inequality follows from the simple fact that, if $X$ can be covered by $k$ balls $B_1,\ldots, B_k$ of radius $r$, then $X-X$ can be covered by  $B_i-B_j$ ($1\leq i,j\leq k$) and hence by
$k^2$ many balls of radius $3r$.

We first assume that  $\alpha_1,\alpha_2$ are $\Q_+$-dependent $(\mbox{mod }1)$.  Since $X$ contains infinitely points, one of $\alpha_1,\alpha_2$ must be irrational. Without loss of generality, we assume that $\alpha_2\not\in \Q$.   Then by the assumption of $\Q_+$-dependence $(\mbox{mod }1)$, one of the following two scenarios must occur: (a) $\alpha_1=\frac{p}{q}\in \Q$;    (b) $\alpha_1\not\in \Q$ and there exist  integers $p_1, p_2, q$ with $p_1, q>0$ so that $\alpha_1=-\frac{p_1}{q}\alpha_2+\frac{p_2}{q}$.

If (a) occurs, since $X$ contains infinitely points, we have $\omega_n=2$ for infinitely many $n$ and hence
$$\bigcup_{j=0}^{q-1}(X+j/q)\supset \{n\alpha_2: \; n\in \N\} \pmod 1.$$
Taking closure and applying the Baire category theorem, we see that $\overline{X}$ has a non-empty interior.

If (b) occurs, since $X$ contains infinitely points, one can check that
either
$$\bigcup_{i=0}^{p_1+q}\bigcup_{j=0}^{q-1}\left(X+\frac{i\alpha_2+j}{q}\right)\supset \{n\alpha_2: \; n\in \N\} \pmod 1,$$
or
$$\bigcup_{i=0}^{p_1+q}\bigcup_{j=0}^{q-1}\left(X+\frac{i\alpha_2+j}{q}\right)\supset \{-n\alpha_2: \; n\in \N\} \pmod 1.$$
Again by the Baire category theorem,  $\overline{X}$ has a non-empty interior.

Next assume that  $\alpha_1$ and $\alpha_2$ are $\Q_+$-independent $(\mbox{mod }1)$. Then  both of them are irrational. Below we treat the two cases separately: (c) $\alpha_2-\alpha_1\in \Q$, (d) $\alpha_2-\alpha_1\not\in \Q$.

First suppose that $\alpha_2-\alpha_1=p/q\in \Q$.  It is easy to see that for $n\geq 1$,  $$x_n\equiv n\alpha_1+p_n/q\pmod 1$$ for some $p_n\in \{0,1,\ldots, q-1\}$.
It follows that
$$\bigcup_{j=0}^{q-1}(X+j/q)\supset \{n\alpha_1: \; n\in \N\} \pmod 1,$$
and so  $\overline{X}$ has a non-empty interior.

Next we consider the case that $\alpha_2-\alpha_1\not\in \Q$.   Suppose that $X-X$ is not dense in $\T$. Then there exists $\delta>0$
 so that  $X-X$ is not $\delta$-dense
in $\T$.

Since $\alpha_2-\alpha_1\notin\Q$, there exists a positive integer $N$ such that
the set  $$\{k(\alpha_2-\alpha_1):\; k=1,\ldots, N\}\pmod 1$$ is
$\delta$-dense in~$\T$.  Write  $\tau(0)=0$ and
$$\tau(n)=\#\{1\leq i\leq n \colon\omega_i=2\} \quad\mbox{ for } n\geq 1,$$
 where $\#A$ stands for the cardinality of $A$.
We claim that
\begin{equation}\label{eq:N}
  \sup_{n,m\in \N}|\tau(n+m)-\tau(n)-\tau(m)|<N.
\end{equation}

Suppose on the contrary  that the claim is false, i.e.,
\begin{equation}
\label{e-ine}
|\tau(n+m)-\tau(n)-\tau(m)|\ge N\quad\text{for some $n, m\in \N$}.
\end{equation}
Fix such $n,m$. Define
 $$b_j=\tau(m+j)-\tau(j),\quad j=0,\dots,n.$$
  Then  $|b_{n}-b_0|\ge N$ by \eqref{e-ine}.  A direct check shows that $$b_{j+1}-b_j=\omega_{m+j+1}-\omega_{j},$$
   which implies  $|b_{j+1}-b_j|\leq 1$.
  Since $|b_{n}-b_0|\ge N$, we see that the set $\{ b_0,\ldots, b_n\}$ contains at least $N$ consecutive integers, say $t+1,\ldots, t+N$.
Observe that for each $k$,
\[x_k\equiv\bigl(k-\tau(k)\bigr)\alpha_1+\tau(k)\alpha_2
 \equiv k\alpha_1+\tau(k)(\alpha_2-\alpha_1)\pmod1.\]
Hence for $j=1,\ldots, n$,
\[\begin{split}
  x_{m+j}-x_j&\equiv m\alpha_1+(\tau(m+j)-\tau(j))(\alpha_2-\alpha_1)\\
  &\equiv m\alpha_1+b_j(\alpha_2-\alpha_1)\pmod1.
\end{split}\]
Therefore,
\[\begin{split}
  X-X&\supset\{x_{m+j}-x_j: \; j=1,\ldots n\}\\
  &\equiv\{m\alpha_1+b_j(\alpha_2-\alpha_1):\; j=1,\ldots n\}\\
 &\supset\{b'+(\alpha_2-\alpha_1),b'+2(\alpha_2-\alpha_1),
  \dots,b'+N(\alpha_2-\alpha_1)\}\pmod1,
\end{split}\]
where $b'=m\alpha_1+t(\alpha_2-\alpha_1)$. Consequently, $X-X$ is
$\delta$-dense in~$\T$, leading to a contraction. This
proves  \eqref{eq:N}.

Next we use \eqref{eq:N}  to show that  $\overline X$ has a non-empty
interior.  Indeed by \eqref{eq:N}, we have
\[\tau(n+m)+N\le(\tau(n)+N)+(\tau(m)+N)\]
and
\[N-\tau(n+m)\le(N-\tau(n))+(N-\tau(m)),\]
that is, the two sequences $(\tau(n)+N)_{n\ge1}$ and $(N-\tau(n))_{n\ge1}$ are
both subadditive. It follows that the limit $\tau=\lim_{n\to\infty}\tau(n)/n$
exists, and  moreover,
\[\tau=\inf_{n\ge1}\frac{\tau(n)+N}n, \quad\text{}\quad
 -\tau=\inf_{n\ge1}\frac{N-\tau(n)}n.\]
That means $|\tau(n)-n\tau|\leq N$ for all $n\ge1$, and so
\begin{equation}\label{eq:nN}
  \bigl|\tau(n)-[n\tau]\bigr|\le N\quad\text{for all $n\ge1$}.
\end{equation}

Set $\tau'=(1-\tau)\alpha_1+\tau\alpha_2$, and let
\[y_n=\{n\tau'\}-\{n\tau\}(\alpha_2-\alpha_1)\;\mbox{(mod 1)}\quad
\text{for $n\ge1$}.\]
Then
\begin{align*}
  y_n &\equiv n((1-\tau)\alpha_1+\tau\alpha_2)-
  \{n\tau\}(\alpha_2-\alpha_1)\\
  &\equiv n\alpha_1+[n\tau](\alpha_2-\alpha_1)\\
  &\equiv n\alpha_1+\tau(n)(\alpha_2-\alpha_1)+z_n\\
  &\equiv x_n+z_n\pmod1,
\end{align*}
where $z_n:=([n\tau]-\tau(n))(\alpha_2-\alpha_1)$. By~\eqref{eq:nN}, for all
$n\ge1$,
\[z_n\in \{k(\alpha_2-\alpha_1)\colon k\in\Z\ \text{and}\ |k|\le N\}=:Z.\]
Let $Y=\{y_n:\; n\in \N\}$; then $Y\subset X+Z\pmod1$. Since $Z$ is finite, by
Baire category theorem, $\overline X$ has a non-empty interior if so does
$\overline Y$.

It remains to show that $\overline Y$ has a non-empty interior. Since $\tau\in [0,1]$,
$\tau$ and $\tau'$ can not be rational numbers simultaneously (otherwise,
 $\alpha_1$ and $\alpha_2$ are not $\Q_+$-independent $(\mbox{mod }1)$). Therefore,
\[W:=\overline{\bigl\{(\{n\tau\},\{n\tau'\})\colon n\ge1\bigr\}}\]
is an infinite compact subgroup of $\T^2=\R^2/\Z^2$. It is either the whole group~
$\T^2$ or finitely many lines in $\T^2$ with rational slope. Notice that
\[\overline Y=\overline{\bigl\{\{n\tau'\}-\{n\tau\}(\alpha_2-\alpha_1)\pmod1
\colon n\ge1\bigr\}},\]
which can be regarded as the image of $W$ under certain projection along an
\emph{irrational} direction since $\alpha_2-\alpha_1\notin\Q$. Consequently,
$\overline Y$ has a non-empty interior and so does $\overline{X}$. This  completes the proof of Theorem \ref{thm-5}(i).
\end{proof}

Before proving Theorem \ref{thm-5}(ii), we first give two simple lemmas.

\begin{lem}\label{l:rational}
Consider the following system of linear equations in the variables
$z_1,\ldots, z_\ell$:
  \begin{equation}\label{eq:esic1}
   \sum_{i=1}^\ell a_{i, j}z_i=b_{j},\quad j=1,2, \ldots,
  \end{equation}
where $a_{i,j}, b_j\in \Q$ for all $i,j$.  Suppose that the system  has a
 real solution. Then it must have a rational
solution.
\end{lem}
\begin{proof}
This is a classical result in linear algebra.
\end{proof}

\begin{lem}
\label{simple} For  $A\subset {\Bbb T}$ and $\delta>0$, let $N_\delta(A)$ denote
the smallest number of intervals of length $\delta$ that are needed to cover
$A$. Then for any positive integer $p$,  we have
$$N_{p\delta}(pA(\mbox{mod }1))\leq N_\delta(A).$$
\end{lem}
\begin{proof}
Suppose that $A$ can be covered by intervals  $I_1,\ldots, I_k$. Then
$pA(\mbox{mod }1)$ can be covered by the intervals  $pI_1(\mbox{mod }1)$,
\ldots, $pI_k(\mbox{mod } 1)$.
\end{proof}

\begin{proof}[Proof of Theorem~\ref{thm-5}(ii)]

First observe that $\dim\mbox{span}_\Q(1,\alpha_1,\ldots, \alpha_\ell)=:1+r>1$, otherwise $\alpha_1,\ldots, \alpha_\ell$ are all rationals and hence $X$ is a finite set,  which leads to a contradiction.  Therefore,  $r\geq 1$.  Pick  a suitable basis $1, \beta_1,\ldots, \beta_r$ of $\mbox{span}_\Q(1,\alpha_1,\ldots, \alpha_\ell)$ so that
\begin{equation}\label{eq:alpahi}
  \alpha_i=\sum_{j=1}^{r}p_{i,j}\beta_j+q_i, \quad i=1,\ldots, \ell,
\end{equation}
for some $p_{i,j}\in \Z$ and $q_i\in \Q$.

For $i=1,\dots,\ell$, set
$$N_i(0)=0, \mbox{  and }\;
N_i(n)=\#\{1\le j\le n\colon\omega_j=i\}\quad
\text{for $n\ge1$}.
$$
Write
\[b_{j}(n)=\sum_{i=1}^{\ell}p_{i,j}N_i(n),\quad 1\leq j\leq r,\; n\geq 0.\]
Then $b_{j}(n)\in \Z$, and moreover,
\begin{equation}
\label{e-bj}
b_j(n+1)-b_j(n)=\sum_{i=1}^\ell p_{i,j}(N_i(n+1)-N_i(n))=p_{\omega_{n+1}, j}.
\end{equation}

 Clearly, we have
 \begin{equation}
 \label{e-xn}
 \begin{split}
 x_n&\equiv\sum_{i=1}^\ell N_i(n)\alpha_i \\
 &\equiv \sum_{i=1}^\ell\left(\Big(\sum_{j=1}^r (p_{i,j}N_i(n)\beta_j\Big) + q_iN_i(n)\right)  \\
 &\equiv \sum_{j=1}^r b_j(n)\beta_j +\sum_{i=1}^\ell q_iN_i(n)\pmod 1.
 \end{split}
 \end{equation}

 As $q_i\in \Q$, the term $c_n:=\sum_{i=1}^\ell q_iN_i(n)\pmod 1$ can take only finitely many different values. However, by assumption,   $x_n$ can take infinitely many different values, thus the sequence $(b_1(n), \ldots, b_r(n))_{n\geq 0}$ of integer vectors is unbounded.
 Therefore, there
exist $r_0\in\{1,\dots,r\}$ and a strictly increasing sequence
$(n_s)_{s\ge1}$ of positive integers such that
\begin{equation}
\label{e-e29}
|b_{r_0}(n_s)|=\max_{1\le j\le r}|b_j(n_s)|\;\; \text{for all $s\ge1$},\;  \mbox{ and }
\lim_{s\to\infty}|b_{r_0}(n_s)|=\infty.
\end{equation}

Choose a positive integer $M$  so that $M>1+\sum_{j=1}^{r}|\beta_j|$.  Then  define $\beta_1^*,\dots,\beta^*_r$ by
$$
\beta^*_j=\left\{
\begin{array}{ll}
\beta_j & \mbox{ if } j\in \{1,\ldots, r\}\backslash \{r_0\},\\
\beta_{r_0}+M & \mbox{ if }j=r_0.
\end{array}
\right.
$$
Correspondingly, set $q_i^*=q_i-Mp_{i,r_0}$ for $1\leq i\leq \ell$.  Clearly $\{1,\beta_1^*,\ldots, \beta_r^*\}$ is still a basis of $\mbox{span}_\Q(1,\alpha_1,\ldots, \alpha_\ell)$ and it satisfies the following relations:
\begin{equation}\label{e-now}
  \alpha_i=\sum_{j=1}^{r}p_{i,j}\beta_j^*+q_i^*, \quad i=1,\ldots, \ell.
\end{equation}

Similar to \eqref{e-xn}, for $n\geq 0$ we have
\begin{equation}
\label{e-xn1}
x_n\equiv \sum_{j=1}^r b_j(n)\beta_j^* +\sum_{i=1}^\ell q_i^*N_i(n)\pmod 1
\end{equation}

Set
\begin{equation}
\label{e-bn}
B(n)=\sum_{j=1}^{r}b_j(n)\beta_j^*
=\sum_{j=1}^{r}\sum_{i=1}^{\ell}p_{i,j}N_i(n)\beta_j^*.
\end{equation}
Then by \eqref{e-e29}, we have
\begin{align*}
|B(n_s)|&=\biggl|\sum_{j=1}^{r}b_j(n_s)\beta_j+b_{r_0} (n_s)M
\biggr|\\
&\ge |b_{r_0}(n_s)| \cdot \left (M-\sum_{j=1}^{r}|\beta_j|\right)\\
&\ge |b_{r_0}(n_s)|.  \end{align*}
Hence,  by \eqref{e-e29} again, we see that
\begin{equation}\label{eq:inf}
  \lim_{s\to\infty}|B(n_s)|=\infty,
\end{equation}
and  the sequence
\begin{equation}\label{eq:bound}
  \left(\frac{b_1(n_s)}{B(n_s)},\dots,
  \frac{b_r(n_s)}{B(n_s)}\right)_{s\ge1}
  \quad\text{is bounded}.
\end{equation}

Now we define a new sequence $(\widetilde{x}_n)_{n\geq 0}$ of points in $\T$ so that  $\widetilde{x}_0=0$ and
\begin{equation}\label{eq:xn`}
  \widetilde{x}_n\equiv B(n)\pmod 1   \quad \mbox{ for $n\ge1$}.
\end{equation}
   By \eqref{e-xn1} and \eqref{e-bn}, we see that
 \begin{equation}
 \label{e-2.16}
 x_n-\widetilde{x}_n\equiv\sum_{i=1}^\ell q_i^*N_i(n)\pmod 1,
\end{equation}
 which can only take finitely many different values.

Next we prove a key lemma about the distribution of the sequence $(\widetilde{x}_n)$.
 \begin{lem}\label{l:kxn}
There exists $k_0\in \N$ such that
  \[
  \sup_{n\ge1}\|k\widetilde{x}_n\|\ge1/5
  \]
for all integers $k\geq k_0$, where $\|x\|=\inf\{|x-z|\colon z\in\Z\}$.
\end{lem}
\begin{proof}
We prove the lemma by  contradiction. Suppose that the lemma  is false. Then
there exists a strictly increasing sequence $(k_l)_{l\ge1}$ of positive
integers so that
\begin{equation}\label{eq:|klxn|<1/5}
  \|k_l\widetilde{x}_n\|<1/5\quad\text{for all $n,l\ge1$}.
\end{equation}

Let $\{x\}$ and $[x]$ denote the fractional part and integer part of the real
number $x$, respectively.

Since the sequence $\bigl(\sum_{j=1}^{r}p_{i,j}\{k_l\beta_j^*\}\bigr)_{l\ge1}$
is bounded for every $i\in\{1,\dots,\ell\}$, by taking a subsequence of
$(k_l)_{l\ge1}$ if necessary, we can assume that
\begin{equation}\label{eq:kb-kb<1/5}
  \biggl|\sum_{j=1}^{r}p_{i,j}\Bigl(\{k_l\beta_j^*\}
  -\{k_m\beta_j^*\}\Bigr)\biggr|<1/5\quad
  \text{for $1\le i\le\ell$ and $l,m\ge1$}.
\end{equation}

For each $l\ge1$, define $y_{l,0}=0$ and
\begin{equation}\label{eq:yln}
  y_{l,n}=\sum_{j=1}^{r}b_j(n)\{k_l\beta_j\}
  =\sum_{j=1}^{r}\sum_{i=1}^{\ell}p_{i,j}N_i(n)\{k_l\beta_j\}
  \quad\text{for $n\geq 1$}.
\end{equation}
By~\eqref{eq:xn`} and \eqref{e-bn}, we have $y_{l,n}\equiv k_l\widetilde{x}_n\pmod1$, and so
$\|y_{l,n}\|<1/5$ by~\eqref{eq:|klxn|<1/5}. We claim that
\begin{equation}\label{yln-ylm<2/5}
  |y_{l,n}-y_{m,n}|<2/5\quad\text{for all $l,m\in\N$ and $n\ge0$}.
\end{equation}

To see it, we proceed by induction on~$n$. Clearly \eqref{yln-ylm<2/5} holds
for $n=0$, since by definition $y_{l,0}=0$ for all $l\ge1$. Now suppose that
$|y_{l,n}-y_{m,n}|<2/5$ for all $l,m\in \N$ and some $n\ge0$. Since
$\|y_{l,n}\|<1/5$ and $\|y_{m,n}\|<1/5$, by \eqref{yln-ylm<2/5} there exists $z\in\Z$ such that
\begin{equation}
\label{e-yln}
y_{l,n},\; y_{m,n}\in(z-1/5,z+1/5).
\end{equation}

Observe that
 \begin{equation}
 \label{e-1/5}
 \begin{split}
& \left|(y_{l,n+1}- y_{l,n})-(y_{m,n+1}-y_{m,n})\right|\\
 &\mbox{}\quad =\left| \sum_{j=1}^r \left(b_{j}(n+1)-b_j(n) \right) \left( \{k_l\beta_j^*\}-\{k_m\beta_j^*\} \right)   \right | \quad \qquad \mbox{(by \eqref{eq:yln})}\\
 &\mbox{}\quad =\left | \sum_{j=1}^rp_{\omega_{n+1}, j} ( \{k_l\beta_j^*\}-\{k_m\beta_j^*\})   \right | \quad \qquad \mbox{(by \eqref{e-bj})}\\
  &\mbox{}\quad \leq 1/5  \quad \qquad \mbox{(by \eqref{eq:kb-kb<1/5})}.
 \end{split}
 \end{equation}
Since $\|y_{l, n+1}\|<1/5$,  we have $|y_{l,n+1}-z'|<1/5$ for some $z'\in\Z$, and so by \eqref{e-yln},
$$|y_{l,n+1}-y_{l,n}-(z'-z)|<2/5.$$
Combining the above inequality with \eqref{e-1/5} yields that $$|y_{m,n+1}-y_{m,n}-(z'-z)|<3/5.$$ Thus, by \eqref{e-yln}, $|y_{m,n+1}-z'|<4/5$. Combining this with $\|y_{m,n+1}\|<1/5$, we have
$|y_{m,n+1}-z'|<1/5$. Consequently, $|y_{l,n+1}-y_{m,n+1}|<2/5$. This
completes the proof of~\eqref{yln-ylm<2/5}.

By~\eqref{eq:yln} and~\eqref{yln-ylm<2/5},
\[
\left|\sum_{j=1}^{r}b_j(n)\left(\{k_l\beta_j^*\}
-\{k_m\beta_j^*\}\right)\right|<\frac25.\]
That is
\[
\left| (k_l-k_m)B(n)-\sum_{j=1}^{r}b_j(n)\left ( [k_l\beta_j^*]
-[k_m\beta_j^*]\right)
\right|<\frac25.
\]
 Replacing $n$ by $n_s$ and
dividing both sides by $|(k_l-k_m)B(n_s)|$  gives
\begin{equation}
\label{e-eend}
\left|
\sum_{j=1}^{r}\frac{b_j(n_s)}{B(n_s)}\cdot
\frac{[k_l\beta_j^*]-[k_m\beta_j^*]}{k_l-k_m}-1
\right|<
\frac{2}{5|(k_l-k_m)B(n_s)|}.
\end{equation}

By \eqref{eq:bound}, the sequence
\[
\left(\frac{b_1(n_s)}{B(n_s)},\dots,
\frac{b_r(n_s)}{B(n_s)}\right)_{s\ge1}\]
is bounded and hence  has an accumulation point, say $(t_1,\ldots, t_r)$.  
By~\eqref{eq:inf} and \eqref{e-eend}, we have
\[\sum_{j=1}^{r}t_j\frac{[k_l\beta_j^*]-[k_m\beta_j^*]}{k_l-k_m}=1
\quad\text{for all distinct $l,m\in\N$}.\]

Since $\frac{[k_l\beta_j^*]-[k_m\beta_j^*]}{k_l-k_m}\in\Q$, by Lemma \ref{l:rational}, there exist
$u_1,\dots,u_r\in\Q$ such that
\[\sum_{j=1}^{r}u_j\frac{[k_l\beta_j^*]-[k_m\beta_j^*]}{k_l-k_m}=1
\quad\text{for all distinct $l,m\in\N$}.\]
Finally, letting $k_l-k_m\to\infty$, we have $\sum_{j=1}^{r}u_j\beta_j^*=1$,
which contradicts the fact that $1,\beta_1^*,\dots,\beta_r^*$ are
$\Q$-independent. This completes the proof of the lemma.
\end{proof}

Let us continue the proof of  Theorem~\ref{thm-5}(ii).  We consider the cases $r=1$ and $r>1$ separately.

First assume that $r=1$. In this case, we show that $K$ has non-empty interior.  For convenience, write  $\beta={\beta_1^*}$ and  $p_i:=p_{i,1}$.
Then $\beta$ is irrational and
\begin{equation}
\label{e-x}
\alpha_i=p_i\beta +q_i^*\quad \mbox{  for } i=1,\ldots, \ell.
 \end{equation}
 Recall that $p_i\in \Z$ and $q_i^*\in \Q$. Pick $q\in \N$ such that all $q_i^*$ are the integral multiples of $1/q$. Let $p=\max_{1\leq i\leq \ell } |p_i|$.
  Since the set  $X=\{x_n:\; n\geq 1\}$ is infinite, we have $p\geq 1$ and moreover,  by the expression \eqref{e-x} of $\alpha_i$, it is not hard to see that
\begin{align*}
\mbox{ either } & \bigcup_{i=-p}^p\bigcup_{j=-q}^q \left(X+i\beta+\frac{j}{q}\right) \supset \{n\beta:\; n\in \N\} \pmod 1\\
\mbox{or }\;\;\; & \bigcup_{i=-p}^p\bigcup_{j=-q}^q \left(X+i\beta+\frac{j}{q}\right) \supset \{-n\beta:\; n\in \N\}\pmod 1.
\end{align*}
 Taking closure and applying the Baire category theorem, we see that $K=\overline{X}$ has a non-empty interior.

Next assume $r\geq 2$.
  Let $m=\max_{1\leq i\leq \ell}\sum_{j=1}^r|p_{i,j}|$. We claim that for every $n\in \N$, there
  exists  $k_n\in\{1,\dots,(mn)^r+1\}$
  such that
  \begin{equation}
  \label{eq:alphabeta}
  \|k_n\beta_j^*\|\leq \frac{1}{mn},\quad j=1,\ldots, r.
  \end{equation}
  To prove this claim, fix $n\in \N$ and partition the unit cube $[0,1]^r$
  into $(mn)^r$ sub-cubes of side length $\frac{1}{mn}$. Consider the
  following $(mn)^r+1$ vectors
  $$v_k=(k\beta_1^*,\dots,k\beta_r^*)\pmod1, \quad k=1,\ldots, (mn)^r+1.
  $$
  By the pigeonhole principle,  two of them, say $v_k$ and $v_{k'}$,  are
  contained in the same  subcube, and  thus $v_k-v_{k'}\in [-\frac{1}{mn},
  \frac{1}{mn}]^r$. Then we have $\|(k'-k)\beta_j^*\|\leq \frac{1}{mn}$ for all
  $j\in\{1,\dots,r\}$. The claim is proved by taking $k_n=|k'-k|$.

  Pick $q\in \N$ such that all $q_i^*$ are the integral multiples of $1/q$. By~\eqref{e-now} and \eqref{eq:alphabeta}, we have
\begin{equation}
  \label{e-kq}
   \|k_nq\alpha_i\|\le\sum_{j=1}^r(q |p_{i,j}|\cdot\|k_n\beta_j^*\|)
   \leq qm\cdot \frac{1}{mn}=\frac{q}{n},
   \quad  i=1,\dots,\ell, \; n\geq 1.
\end{equation}
  Define $y_{n,s}\in \T$ so that
  \begin{equation}
  \label{e-ys}
  y_{n, s}\equiv k_nqx_s \pmod 1,\quad n\geq 1,\; s=0,1,\ldots,
  \end{equation}
   and let $Y_n=\{y_{n,s}:\;
  s=0,1,\ldots\}\subset \T$. By \eqref{e-kq} and the definition of $x_s$, we have
  $\|y_{n,s+1}-y_{n,s}\|\leq q/n$ for each $s\geq 0$.  It follows that
  $$
  I_n:=\bigcup_{s\geq 0} \left[y_{n,s}-\frac{q}{2n}, \; y_{n,s}+\frac{q}{2n}\right] \pmod1
  $$
  is an interval in $\T$ containing $y_{n,0}=0$.

  By \eqref{e-2.16}, we have $qx_n=q\widetilde{x}_n\pmod 1$ for each $n\geq 1$.  Therefore, by Lemma~\ref{l:kxn}, there exists $k_0>0$ such that
    $$a:=\inf_{k\geq k_0} \sup_{s\geq 0}\|kqx_s\|=\inf_{k\geq k_0} \sup_{s\geq 0}\|kq\widetilde{x}_s\|\geq \frac{1}{5}.$$
  Hence   by  \eqref{e-ys},  for any  $n$ so that $k_n>k_0$, we have $\sup_{s\geq 0} \|y_{n,s}\|\geq a>0$, and hence   the length of $I_n$ is
  not less than $a$. It follows that
  $$
  N_{q/n}(Y_n)\geq{an/q},
  $$
  where $N_\delta(A)$ stands for the smallest number of intervals of length
  $\delta$ that are needed to cover $A$.  Since $Y_n=k_nqX\;(\mbox{mod
  }1)$, by Lemma \ref{simple}, we have $$N_{1/(nk_n)}(X)\geq
  N_{q/n}(Y_n)\geq {an/q}.$$ Since $k_n\leq
  (mn)^r+1$, we have
  $$
  N_{1/(2m^r n^{r+1})}(X)\geq N_{q/n}(X)\geq{an/q}.
  $$
  Noticing that the above inequality holds for all $n\in \N$ and $m,q,r$ are
  constant, we have
  \[\lbdim X\ge\liminf_{n\to\infty}\frac{\log({an/q})}{\log (2m^r n^{r+1})}=
  \frac1{r+1}.\]
  Thus we have $\lbdim K= \lbdim X\geq 1/(r+1)$.
\end{proof}

\section{The proof of Theorem~\ref{thm-2}}\label{sec:proof}

 We begin with a lemma about orthogonal groups. Let $\mathcal O(d)$ be    the group of $d\times d$ orthogonal matrices operated by  matrix multiplication.
\begin{lem}
\label{lem-orth}
  For every  $P\in\mathcal O(d)$, there exists  $k\in\N$ such that the closure
  of $\{P^{kj}\colon j\ge0\}$ in $\mathcal O(d)$ is a connected subgroup of $\mathcal O(d)$.
\end{lem}
\begin{proof}
This result might be well known, however we are not able to find a
reference, so a proof is included for the reader's convenience.

Let  $P\in \mathcal O(d)$, and let  $W$ be the closure of $\{P^j\colon j\ge 0\}$ in  $\mathcal O(d)$. It is not hard to see that $W$ is a
compact Abelian subgroup of~$\mathcal O(d)$. Hence by the Cartan theorem
(cf.~\cite[Theorem~3.3.1]{Price77}), $W$ is also a Lie group. Let $W_0$ be
the connected component of $W$ containing the unit element $I$. Then $W_0$ is
a closed normal subgroup of $W$, and it is also open in $W$
(cf.~\cite[Lemma~2.1.4]{Price77}). By the finite covering theorem, $W$ has
only finitely many connected branches. It follows that  the quotient group $W/W_0$
is finite.

Let $\Z_0=\{j\in\Z\colon P^j\in W_0\}$.  Then $\Z_0$  is a subgroup of $\Z$. Since $W/W_0$ is finite, there are
distinct $j_1,j_2\in\Z$ such that $P^{j_1}$ and $P^{j_2}$ both belong to a
  coset of~$W_0$. Hence $P^{j_2-j_1}\in W_0$, and  consequently, $\Z_0$   contains a
nonzero element $j_2-j_1$. Therefore, $\Z_0= k\Z$ for some
$k\ge1$. We claim that $W_0$ is the closure of $\{P^{kj}\colon j\ge0\}$,  from which
 the lemma follows since $W_0$ is connected.

Clearly  $W_0$ contains the closure of $\{P^{kj}\colon j\ge0\}$.
Conversely, since $W_0$ is open and disjoint from $\{P^j\colon k\nmid j\}$,
it is also disjoint from the closure of $\{P^j\colon k\nmid j\}$. Thus, $W_0$
is contained  in the closure of $\{P^{kj}\colon j\ge0\}$.  This completes the proof of the lemma.
\end{proof}

\begin{proof}[Proof of Theorem~\ref{thm-2}]

For brevity, we write $\phi_I=\phi_{i_1}\circ\dots\circ\phi_{i_n}$ and
$\rho_I=\rho_{i_1}\dotsb\rho_{i_n}$ for $I=i_1\dots
i_n\in\{1,\dots,\ell\}^n$. Similarly, we also use the abbreviations $\psi_J$
and $\gamma_J$ for $J\in\{1,\dots,m\}^n$.

 Since $F$ can be affinely embedded into $E$, there exist an invertible real $d\times d$ matrix $M$ and $b\in {\Bbb R}^d$ such that
\begin{equation}\label{eq:MFsubsetE}
  M(F)+b\subset E.
\end{equation}

 Without loss of generality, we only prove that the conclusion of
Theorem~\ref{thm-2} holds for $j=1$, that is, there exist non-negative
rational numbers $t_{1,i}$, $i=1,\ldots, \ell$, such that
$$\gamma_1=\prod_{i=1}^\ell \rho_i^{t_{1,i}}.$$
This is equivalent to show that $\alpha_1,\dots,\alpha_\ell$ are not
$\Q_+$-independent   $(\mbox{mod }1)$, where $$\alpha_i:=-\frac{\log
\rho_i}{\log \gamma_1}\quad\text{for $i\in\{1,\dots,\ell\}$}.$$

 Let $P_1$ be the orthogonal part of $\psi_1$.  By Lemma \ref{lem-orth}, there exists $l\in \N$ such that the closure of $\{P_1^{lj}: j\geq 0\}$ in $\mathcal O(d)$ is a connected subgroup of $\mathcal O(d)$.
In what follows,  replacing $\psi_1$ by $\psi_1^l$ if necessary, we may always assume that  the closure  $\{P_1^{j}: j\geq 0\}$ in $\mathcal O(d)$ is connected.

Let $x$ be the fixed point of $\psi_1$. Then $x\in\psi_1^n(F)$ for any integer $n\geq 0$.   By \eqref{eq:MFsubsetE}, we have $$y:=M(x)+b\in E,$$ and thus  there exists a symbolic coding
$i_1i_2\cdots\in\{1,\dots,\ell\}^\N$ such that
\begin{equation}\label{eq:xcoding}
   y=\lim_{n\to\infty}\phi_{i_1\dots i_n}(0).
\end{equation}
 Clearly $y\in \phi_{i_1\ldots i_n}(E)$ for each $n\geq 0$, which implies that
\begin{equation}
\label{e-Mpsi}
 (M(\psi_1^k(F))+b)\cap \phi_{i_1\ldots i_n}(E)\neq \emptyset\quad \mbox{ for any } k,n\geq 0.
 \end{equation}

 Since $\Phi$ satisfies the strong separation condition, we have
\begin{equation}
\label{e-delta}
\delta:=\min_{i\ne j}\dist\bigl(\phi_i(E),\phi_j(E)\bigr)>0.
\end{equation}
Moreover, for each $n\in \N$,  we have
\begin{equation}
\label{e-gap}
\dist\bigl(\phi_{i_1\dots i_n}(E),\; E\setminus\phi_{i_1\dots i_n}(E)\bigr)\geq \rho_{i_1\dots i_{n-1}}\delta>0.
\end{equation}

 For $k, n\geq 0$,  by \eqref{e-Mpsi} and \eqref{e-gap} we have
\begin{equation}
\label{e-3.5}
M(\psi_1^k(F))+b \subset \phi_{i_1\dots i_n}(E)
\quad \mbox{
if } \quad\mbox{diam}((M(\psi_1^k(F))<\rho_{i_1\dots i_{n-1}}\delta.
\end{equation}

  Now for $n \geq 1$, define
\begin{equation}\label{eq:sn}
  s_n=\min\bigl\{k\ge0\colon M(\psi_1^k(F))+b\subset
  \phi_{i_1\dots i_n}(E)\bigr\}.
\end{equation}
 Then by \eqref{e-3.5}, $s_n<\infty$.
Write
\begin{equation}\label{eq:Mratio}
\begin{split}
  \|M\|&=\max\{|Mv|\colon\text{$v\in\R^d$ with $|v|=1$}\},\\
     \sm M\sm &=\min\{|Mv|\colon\text{$v\in\R^d$ with $|v|=1$}\},
\end{split}
\end{equation}
 where $|\cdot|$ denotes the standard Euclidean norm.

By \eqref{eq:sn}-\eqref{eq:Mratio}, we have
\[\sm M\sm \gamma_1^{s_n}\diam F \le \diam M(\psi_1^{s_n}(F)) \le
\diam\phi_{i_1\dots i_n}(E)=\rho_{i_1\dots i_n}\diam E.\]
Thus, we have
\begin{equation}\label{eq:EFub}
  \frac{\gamma_1^{s_n}}{\rho_{i_1\dots i_n}}\le
  \frac{\diam E}{\sm M \sm \diam F}
  \quad\text{for all $n\ge1$}.
\end{equation}
For the lower bound, we claim that
\begin{equation}\label{eq:EFlb}
  \frac{\gamma_1^{s_n}}{\rho_{i_1\dots i_n}}
  \ge\frac{\gamma_1\delta}{\rho^*\|M\|\diam F}\quad\text{if $s_n\ge1$},
\end{equation}
where  $\delta$ is defined as in \eqref{e-delta} and
$\rho^*:=\max_{1\le i\le\ell}\rho_i$.  Indeed, suppose that
 \eqref{eq:EFlb} fails  for some $n$ with $s_n\geq 1$. Then
\begin{align*}
  \diam M(\psi_1^{s_n-1}(F))&\le\|M\| \gamma_1^{s_n-1}\diam F
  <(\rho^*)^{-1}\rho_{i_1\dots i_n}\delta\le\rho_{i_1\dots i_{n-1}}\delta.
\end{align*}
 By \eqref{e-3.5}, $M(\psi_1^{s_n-1}(F))+b\subset\phi_{i_1\dots i_n}(E)$,
which contradicts the definition of~$s_n$. This completes the proof of \eqref{eq:EFlb}.

For $1\le i\le\ell$, let $O_i$  be the orthogonal
 part of $\phi_i$. From
$M(\psi_1^{s_n}(F))+b\subset\phi_{i_1\dots i_n}(E)$  we have
$$
(\phi_{i_1\cdots i_n})^{-1}(M(\psi_1^{s_n}(F))+b) \subset E.$$
Hence
\[\rho_{i_1\dots i_n}^{-1}\gamma_1^{s_n}Q_n(F)+b_n\subset E\]
for some $b_n\in \R^d$, where $Q_n=(O_{i_1}\circ\cdots\circ O_{i_n})^{-1}MP_1^{s_n}$.
 Taking  algebraic difference,
we have
\begin{equation}
\label{e-incl}
\rho_{i_1\dots i_n}^{-1}\gamma_1^{s_n}Q_n(F-F)\subset E-E,\quad n\geq 1.
\end{equation}

 Fix a nonzero vector $v\in F-F$.  For any integer $k\geq 0$, we have $$\gamma_1^k P_1^k v\in \psi_1^k(F)-\psi_1^k(F)\subset F-F.$$

   Hence by \eqref{e-incl},
\begin{equation*}
\label{e-cont}
\rho_{i_1\dots i_n}^{-1}\gamma_1^{s_n+k}Q_n(P_1^kv)\in
 E-E,\quad \forall\; n\geq 1, \; k\geq 0.
 \end{equation*}
Taking norm on both sides yields
\begin{equation}
\label{e-norm}
\rho_{i_1\dots i_n}^{-1}\gamma_1^{s_n+k}|MP_1^{s_n+k}v|\in \{|x_1-x_2|\colon x_1,x_2\in E\}, \quad \forall \; n\geq 1,\; k\geq 0.
\end{equation}

 Next we continue our arguments according to whether the sequence $\left(|MP_1^jv|\right)_{j=0}^\infty$ is
constant.

{\textit{Case~(i)}}:  the sequence $(|MP_1^jv|)_{j=0}^\infty$ is constant.

In this case, applying \eqref{e-norm} with $k=0$ we obtain
\[U:=\{|x_1-x_2|\colon x_1,x_2\in E\}\supset
V:=\bigl\{\rho_{i_1\dots i_n}^{-1}\gamma_1^{s_n}a\colon
 n\ge1\bigr\},\]
 where $a$ is the positive constant $|MP_1^jv|$.
Set  $b_*=\inf V$ and $b^*=\sup V$. By  \eqref{eq:EFub}-\eqref{eq:EFlb},  $0<b_*<b^*<\infty$.

 Define $f:\;[b_*, b^*]\to \T$ by  $f(t)=\log t/\log\gamma_1\pmod1$. Notice that  $[b_*,b^*]$ can be
written as the union of finitely many disjoint subintervals of the form
$[b_*,b^*]\cap[\gamma_1^{n+1},\gamma_1^n)$ with $n\in\Z$, and  restricted  on each
non-empty interval $[b_*,b^*]\cap[\gamma_1^{n+1},\gamma_1^n)$, $f$ is Lipschitz. Hence
we have
\begin{equation}\label{eq:dimension}
  \ubdim f(V)\le\ubdim V\le\ubdim U\le\ubdim(E-E)
  \le\ubdim E\times E=2\hdim E.
\end{equation}
where $\ubdim$ stands for upper box-counting dimension
(cf.~\cite{Fal-book}). Recall that $\alpha_i=-\log\rho_i/\log\gamma_1$ for
$1\le i\le\ell$. Clearly,
\begin{equation}
\label{e-e1}
\dim \mbox{span}_\Q(\alpha_1,\ldots, \alpha_\ell)=
\dim \mbox{span}_\Q(\log \rho_1,\ldots, \log \rho_\ell)=:\lambda.
\end{equation}

Let $\omega=i_1i_2\dotsc\in\{1,\dots,\ell\}^\N$, where $i_1i_2\dotsc$ is the
symbolic coding of~$y$ (see~\eqref{eq:xcoding}). Define a sequence  $(x_n(\omega))_{n=1}^\infty \subset \T$ so that
$$
x_n(\omega)\equiv \sum_{k=1}^n\alpha_{i_k}\pmod1 \;\;\mbox{ for } n\geq 1. $$
Set $X(\omega)=\{x_n(\omega)\colon n\in\N\}$.
 Then we have
\[f(V)\supset X(\omega)+\frac{\log a}{\log\gamma_1}\pmod1.\]
 Combining this with \eqref{eq:dimension} yields
\begin{equation}\label{eq:EXdim}
  \hdim E\ge  (1/2)\ubdim X(\omega).
\end{equation}

 Now suppose on the contrary that  $\alpha_1,\ldots,\alpha_\ell$ are
$\Q_+$-independent $(\mbox{mod }1)$. Notice that $\overline{X(\omega)}$ is
an $(\alpha_1,\cdots,\alpha_\ell)$-set. By Corollary~\ref{cor-1}, we have
$$\lbdim \overline{X(\omega)} \ge \left\{
\begin{array}
{ll}
1/2, &\mbox{ if }\ell=2,\\
1, &\mbox{ if }\ell\ge3,\ r=1,\\
1/(r+1), &\mbox{ if }\ell\ge3,\ r>1,\\
\end{array}
\right.
$$
where $r=\dim\mbox{span}_\Q(1,\alpha_1,\ldots, \alpha_\ell)-1$.
By~\eqref{e-e1}, $\lambda=\dim\mbox{span}_\Q(\alpha_1,\ldots,
\alpha_\ell)\ge r$.

 Hence
by~\eqref{eq:EXdim}, we have
\[\hdim E\ge\frac12 \ubdim  X(\omega)=\frac12 \ubdim \overline{X(\omega)} \ge
\begin{cases}
  1/4, & \mbox{if } \ell=2, \\
  1/2, & \mbox{if } \ell\ge3,\ \lambda=1, \\
  1/(2\lambda+2), & \mbox{if }\ell\ge3,\ \lambda>1.
\end{cases}\]
Therefore, $\hdim E\ge c$, where $c$ is given as in \eqref{e-1.1}. It
contradicts the assumption that $\hdim E<c$. This completes the proof of
Theorem~\ref{thm-2} in Case~(i).

{\textit{Case~(ii)}}: the sequence  $(|MP_1^jv|)_{j=0}^\infty$ is not  constant.

  For any integer $p\geq s_1$, let $n=n_p$ be the largest integer  so that $s_n\leq p$, and
define
\begin{equation}
\label{e-u1u2}
u_{1,p}=\rho_{i_1\dots i_n}^{-1}\gamma_1^{p}Q_nP_1^{p-s_n}v,\qquad u_{2,p}=\rho_{i_1\dots i_n}^{-1}\gamma_1^{p+1}Q_nP_1^{p+1-s_n}v;
\end{equation}
taking $k=p-s_n$ and $p-s_n+1$ in \eqref{e-norm} respectively, we have
\begin{equation}
\label{e-u1p}
u_{1,p}, \; u_{2, p}\in E-E.
\end{equation}
By \eqref{e-u1u2}, we have
\begin{equation}\label{eq:pgv}
 \frac{|u_{2,p}|}{\gamma_1|u_{1,p}|}= \frac{|MP_1^{p+1}v|}{|MP_1^pv|}\quad \mbox{ for all }p\geq s_1.
\end{equation}
Furthermore, by \eqref{eq:EFub}-\eqref{eq:EFlb}, there exist two positive constants $c_1, c_2$ so that
\begin{equation}
\label{e-bound}
|u_{1,p}|, \; |u_{2, p}|\in [c_1, c_2]\quad \mbox{ for all }p\geq s_1.
\end{equation}

Now let  $W$ denote the closure  of $\{P_1^p\colon p\ge 0\}$ in $\mathcal O(d)$. As we have assumed,   $W$ is a connected  subgroup of  $\mathcal O(d)$.

 Write
 $$
 U^*=\{|x_1-x_2|\colon x_1,x_2\in E\}\cap [c_1, c_2].
 $$
 Define
\[\pi_1\colon U^*\times U^*\to\R,\
(u_1,u_2)\mapsto\frac{u_2}{\gamma_1u_1}\]
and
\[\pi_2\colon W\to\R,\ g\mapsto\frac{|MP_1gv|}{|Mgv|}.\]
It is clear that $U^*$ is a compact subset of $[c_1, c_2]$ with $c_1>0$,  thus $\pi_1$ is  Lipschitz and $\pi_1(U^*\times U^*)$ is compact.  Moreover, $\pi_2$ is continuous.
By~\eqref{e-u1p}-\eqref{e-bound} and noting that $W$ is also the closure of $\{P^p_1:\; p\geq s_1\}$,  we have
\begin{equation}\label{eq:pi2Wpi1U}
  \pi_2(W)\subset\pi_1(U^*\times U^*).
\end{equation}

 We claim that $\pi_2$ is not a constant function.
Otherwise, suppose that $$\frac{{|MP_1gv|}}{{|Mgv|}}=a$$ for all $g\in W$. We have
$a\ne1$ since the sequence $(|MP_1^pv|)_{p=0}^\infty$ is not  constant. If $a<1$, then
$|MP_1^pv|\to0$ as $p\to\infty$, and so $|Mgv|=0$ for some $g\in W$. This is
impossible since $M$ is invertible. If $a>1$, then $|MP_1^pv|\to\infty$ as
$p\to\infty$. This is also impossible since $|P_1^pv|=|v|$ for all $p\ge0$.

 Due to the above claim and the connectedness of  $W$,
the set $\pi_2(W)$ is connected and contains at least two different elements, hence it
 is a non-degenerate  interval. Therefore by \eqref{eq:pi2Wpi1U},
\[4\hdim E\ge\hdim U^*\times U^*\ge\hdim\pi_1(U^*\times U^*)\ge
\hdim\pi_2(W)=1.\]
Thus, $\hdim E\ge1/4\ge c$,  a contradiction again. Therefore Case (ii) can not occur.   This completes the proof of Theorem~\ref{thm-2}.
\end{proof}

\section{Final  questions}
\label{sec:4}
Here we pose several questions about Theorem \ref{thm-5}:

\begin{itemize}
\item[(Q1)] The lower bounds given in Theorem \ref{thm-5}   on the lower box-counting dimension of  $(\alpha_1,\ldots,\alpha_\ell)$-orbits might not be sharp. Are there any better or optimal bounds? How about the packing dimension of the closure of these sets? \footnote{In Theorem \ref{thm-5}(i),  since $\dim_H(K-K)=1$,  by \cite[Theorem 3]{Tri82} we have $\dim_PK\geq 1/2$.}
\item[(Q2)] It is easy to see that  Theorem \ref{thm-5} can be extended to  high dimensional tori. Is it possible to extend the result to  general compact Lie groups?
\end{itemize}

\subsection*{Acknowledgements}

The authors are grateful to Wen Huang for  helpful discussions and to Pablo
Shmerkin for pointing out the reference~\cite{Kat79}.
The first author was partially supported by the HKRGC GRF grants (projects CUHK401013, CUHK14302415). The second author was partially supported by NSFC (11471124 and 11371148).


\end{document}